\documentclass[leqno]{article}
\usepackage{amsmath}
\usepackage[latin1]{inputenc}
\usepackage{amssymb}
\usepackage{xcolor}
\usepackage{amsthm} 
\usepackage{enumerate}                                        
\usepackage{url}
\usepackage{graphicx}
\usepackage{epsfig}
\newtheorem{theorem}{\bf Theorem}

\newtheorem{corollary}[theorem]{\bf Corollary}
\newtheorem{lemma}[theorem]{\bf Lemma}

\newtheorem{claim}[theorem]{\bf Claim}
\newtheorem{proposition}[theorem]{\bf Proposition}
\newtheorem{definition}[theorem]{\bf Definition}

\newtheorem{remark}[theorem]{\bf Remark}

\numberwithin{equation}{section}
\numberwithin{theorem}{section}
\numberwithin{figure}{section}

\def\R{\mathbb{R}}

\def\R{\mathbb{R}}

\begin{document}
\renewcommand{\thefootnote}{}
\footnotetext{Research partially supported by Ministerio de Educaci\'on Grants No: MTM2013-43970-P  and partially supported by the Coordena\c{c}\~ao de Aperfei\c{c}oamento de Pessoal de N\'ivel Superior - Brasil (CAPES) - Finance Code 001} 

\title{Mean convex properly embedded  $[\varphi,\vec{e}_{3}]$-minimal surfaces in $\R^3$} 
\author{Antonio Mart\'inez$^1$, A.L. Mart\'inez-Trivi\~no$^1$, J. P. dos Santos$^2$}
\vspace{.1in}
\date{}
\maketitle
{
\noindent $^1$Departamento de Geometr\'\i a y Topolog\'\i a, Universidad de Granada, E-18071 Granada, Spain\\ 
\noindent $^2$ Departamento de Matem\'atica, Universidade de Bras\'\i lia, Bras\'\i lia-DF, Brazil \\ \\
e-mails: amartine@ugr.es, aluismartinez@ugr.es, joaopsantos@unb.br}
\begin{abstract}
We establish curvature estimates and a convexity result for mean convex properly embedded $[\varphi,\vec{e}_{3}]$-minimal surfaces in $\mathbb{R}^3$, i.e., $\varphi$-minimal surfaces when $\varphi$ depends only on the third coordinate of $\mathbb{R}^3$. Led by the works on curvature estimates for surfaces in 3-manifolds, due to White for minimal surfaces, to Rosenberg, Souam and Toubiana, for stable CMC surfaces, and to Spruck and Xiao for stable translating solitons in $\mathbb{R}^3$, we use a compactness argument to provide curvature estimates for a family of mean convex $[\varphi,\vec{e}_{3}]$-minimal surfaces in $\mathbb{R}^{3}$. We apply this result to generalize the convexity property of Spruck and Xiao for translating solitons. More precisely, we characterize  the convexity of a  properly embedded   $[\varphi,\vec{e}_{3}]$-minimal surface in $\mathbb{R}^{3}$ with non positive mean curvature when   the growth at infinity of $\varphi$ is at most quadratic.
\end{abstract}
\vspace{0.2 cm}


\everymath={\displaystyle}

\section{Introduction.}
From a physical point of view, a $\varphi$-minimal surface $\Sigma$ in a domain $\Omega$ of $\mathbb{R}^{3}$ arises as a surface in equilibrium under a conservative force field ${\cal F}$, with potential $e^{\varphi}$ for some smooth function $\varphi$ on  $\Omega\subseteq\mathbb{R}^{3}$ (see \cite[pp. 173-187]{Poisson}). It can be also viewed (see \cite{Ilm94}) either as a critical point of the weighted area functional

\begin{equation}
\label{areafunct}
\mathcal{A}^\varphi(\Sigma)=\int_{\Sigma}e^{\varphi}\, d\Sigma,
\end{equation}
where $d\Sigma$ is the area element of $\Sigma$, or as a minimal surface  in $\mathbb{R}^{3}$ with the following conformally changed metric

\begin{equation}
\label{metricIlm}
\langle\cdot,\cdot\rangle^\varphi:=e^{\varphi}\langle\cdot,\cdot\rangle.
\end{equation}

When $\varphi$ only depends on the third coordinate in $\mathbb{R}^{3}$,  $\Sigma$ will be called $[\varphi,\vec{e}_{3}]$-\textit{minimal surface} and then, the equilibrium condition  is given in terms of  the mean curvature vector $\textbf{H}$ of $\Sigma$ as follows
\begin{equation}
\label{meancurv}
\textbf{H}=\dot{\varphi}\,\vec{e}_{3}^{\perp},
\end{equation}
where $(\dot{\ \ })$ denotes the derivative with respect to the third coordinate and $\perp$ is the projection to the normal bundle of $\Sigma$. The condition   \eqref{meancurv} is the equation of a heavy surface in a gravitational field $ {\cal F}=(0,0,g  \ {\cal E}(z))$,  where $g$ is the  gravitational constant and ${\cal E}(z)$ a density function on  the surface.

Minimal surfaces are obtained for  $\dot{\varphi}\equiv 0$. The case $\dot{\varphi}= constant\neq 0$ gives translating solitons, that is, surfaces $\Sigma$   such that $ t\mapsto \Sigma + t \vec{e}_3$ is a mean curvature flow. The case $\dot{\varphi}(z) = \frac{\alpha}{z}$, $z>0$, $\alpha \in \R$,  includes the familiy of ``perfect domes'' (when $ \alpha=1$, see \cite{BHT}) and the family of minimal surfaces in the hyperbolic space $\mathbb{H}^3$ (when $\alpha=-2$).

A $[\varphi,\vec{e}_{3}]$-minimal surface   $\Sigma$ is said to be \textit{stable} if it is stable as minimal surface in the Ilmanen's space $\Omega^\varphi:=(\Omega,\langle\cdot,\cdot\rangle^\varphi)$, that is (see \cite[Appendix]{CMZ}), if for any compactly supported smooth function $v\in\mathcal{C}_{0}^{\infty}(\Sigma)$, it holds that
\begin{equation}
\label{destability}
\int_{\Sigma}\, e^{\varphi}\left( \vert\nabla v\vert^{2}-(\vert\mathcal{S}\vert^{2}-\ddot{\varphi}\eta^{2})v^{2}\right)\, d\Sigma\geq 0,
\end{equation}
where $\nabla$ denotes the gradient operator, $|\mathcal{S}|^2=H^2-2K$ denotes the length of the second fundamental form and $H$ and $K$ stand for the mean and Gaussian curvatures of  $\Sigma$ in $\R^3$. Furthermore, we denote by $\eta:=\langle N,\vec{e}_{3}\rangle$ the angle function respect to the direction $\vec{e}_{3}$. 

\

In 1983, Schoen \cite{SC} obtained an estimate for the length of the second fundamental form of stable minimal surfaces in a 3-manifold. In particular, in $\mathbb{R}^{3}$, he proved the existence of a constant $C$ such that
$$\vert\mathcal{S}(p)\vert\leq\frac{C}{d_{\Sigma}(p,\partial\Sigma)}, \ \ p\in\Sigma,$$
for any stable minimal surface $\Sigma$ in $\mathbb{R}^{3}$ where $d_{\Sigma}$ stands for  the intrinsic distance of $\Sigma$.
Later,  in 2010, Rosenberg, Souam and Toubiana \cite{RST} obtained an estimate for the length of the second fundamental form, depending on the distance to the boundary, for any stable $H$-surface $\Sigma$ in a complete Riemannian $3$-manifold of bounded sectional curvature $\vert\mathbb{K}\vert\leq\beta < +\infty$. More precisely, they proved the existence of a constant $C>0$ such that
$$\vert\mathcal{S}(p)\vert\leq\frac{C}{\text{min}\{d_{\Sigma}(p,\partial\Sigma),\pi/2\sqrt{\beta}\}}, \ \ p\in\Sigma.$$
More recently, in 2016, White \cite{W1} obtained an estimate  for the length of the second fundamental form for minimal surfaces with finite total absolute curvature less than $4\pi$ in $3$-manifolds, depending of the distance to the boundary, of the sectional curvature and of the gradient of the sectional curvature of the ambient space. 

\

Following Colding and Minicozzi  method, \cite{CM,CM1}, Spruck and Xiao \cite{SX} have also obtained area and curvature bounds for  complete mean convex translating solitons in $\mathbb{R}^{3}$. As application and  using the Omori-Yau Theorem (see, for example, \cite{AMR}) they have proved  one of the fundamental results in the recent development of translating solitons theory conjectured by Wang in \cite{Wang}:
\renewcommand{\thetheorem}{\arabic{theorem}}
\begin{theorem}{{\rm \cite[Theorem 1.1]{SX}}}
Let $\Sigma \subset \R^3$ be a complete immersed two-sided translating soliton with nonnegative mean curvature. Then $\Sigma$ is convex. \label{Thm1.1-SX}
\end{theorem}

In this paper we extend the results in \cite{SX} to  mean convex $[\varphi,\vec{e}_{3}]$-minimal surfaces with empty boundary in $\R^3_\alpha=\{p\in\R^3\ \vert \ \langle p,\vec{e}_{3}\rangle>\alpha\}$, where $\varphi: \R \longrightarrow \R$ is a smooth function  satisfying
\begin{equation} \dot{\varphi} >0, \quad \ddot{\varphi}\geq 0 \quad \text{on $]\alpha,\infty[$}. \label{c1}
\end{equation}
More precisely, we see that mean convex examples are stable (Proposition \ref{stableHpositiva}), and  we prove area  estimates (Theorem \ref{boundnessarea}) when
\begin{equation}
\Gamma :=\sup_{]\alpha,+\infty[}  (2 \ddot{\varphi} - \dot{\varphi}^2 )< +\infty. \label{c2}
\end{equation}
To obtain curvature bounds, we need a good control at infinity of the function $\varphi$. To be more precise, we are going to consider that   $ z \mapsto \frac{\dot{\varphi}(z)}{z}$ is analytic at $+\infty$; i.e.,
$\dot{\varphi}$ has the following series expansion at $+\infty$:
\begin{align}
& \dot{\varphi}(u)= \Lambda \, u+\beta+\sum_{i=1}^{\infty}\frac{c_{i}}{u^{i}}, \quad \text{$u$ large enough},\label{cc3}
\end{align}
with $\Lambda\geq 0$ and $\beta>0$ if $\Lambda=0$.

It is worth to note that  condition \eqref{cc3} implies \eqref{c2}. Apart of a natural extension of the best known examples,  conditions \eqref{c1} and \eqref{cc3}  are interesting because under these assumptions it is possible to know explicitly the asymptotic behavior of rotational and translational invariant examples (see \cite{MM}).  

The main results obtained in this paper can be summarized in the following two theorems:

\setcounter{theorem}{0}
\renewcommand{\thetheorem}{\Alph{theorem}}
\begin{theorem}\label{mt1} Let $\Sigma$ be a  properly embedded $[\varphi,\vec{e}_{3}]$-minimal surface in $\R^3_\alpha$  with non positive mean curvature, locally bounded genus and  $\varphi: \R \longrightarrow \R$  satisfying \eqref{c1} and \eqref{cc3}. Then $\vert\mathcal{S}\vert/\dot{\varphi}$ is bounded on  $\Sigma$. In particular,  if  $\Lambda=0$,  $\vert\mathcal{S}\vert$ is  bounded and if $\Lambda\neq 0$, $\vert\mathcal{S}\vert$ may go to infinity but with at most a linear growth in hight.  
\end{theorem}

\begin{theorem}\label{mt2} Let $\Sigma$ be a  properly embedded $[\varphi,\vec{e}_{3}]$-minimal surface in $\R^3_\alpha$  with non positive mean curvature, locally bounded genus and  $\varphi: \R \longrightarrow \R$  satisfying \eqref{c1},  \eqref{cc3} and $\dddot{\varphi}\leq 0$ on $]\alpha,+\infty[$. Then $\Sigma$ is convex if and only if the function $\Lambda K$ is bounded from bellow.
\end{theorem}

\renewcommand{\thetheorem}{\arabic{theorem}}
\numberwithin{theorem}{section}
The paper is organized as follows. In Sections 2 and 3, we show some facts about the geometry of the Ilmanen's space, introduced in \cite{Ilm94} and give some notations and fundamental equations of $[\varphi,\vec{e}_{3}]$-minimal surfaces. 
Following a similar approach as in \cite{SX} and using a compactness result of White, \cite[Theorem 2.1]{W2}, we obtain a  blow-up theorem for $[\varphi,\vec{e}_{3}]$-minimal  which allow us to prove  Theorem \ref{mt1}. This is needed for the proof of Theorem \ref{mt2} in section 5 which is based on   a generalized Omori-Yau's maximum principle (see \cite[Theorem 3.2]{AMR}).

\

\noindent {\bf Acknowledgements}: The
authors are grateful to  Jos\'e Antonio G\'alvez   for helpful comments during the preparation of this manuscript.

\section{Geometry of the Ilmanen's space.}
Let $\varphi:I\subseteq\mathbb{R}\rightarrow\mathbb{R}$ be  a smooth function defined in an open  interval $I$ of $\mathbb{R}$. Following \cite{Ilm94}, consider  the Ilmanen's space $\Omega^\varphi$  as the Riemannian manifold $\Omega=\mathbb{R}^{2}\times I $ with the Euclidean conformal metric
$\langle\cdot,\cdot\rangle^{\varphi}$ defined at any point $p=(x_{1},x_{2},x_{3})\in \Omega$ by 
$$\langle\cdot,\cdot\rangle^{\varphi}_p=e^{\varphi (x_{3})}\langle\cdot,\cdot\rangle_p,$$ 
\noindent
Denote by $D$  and $R$ (respectively, $D^{\varphi}$ and $R^{\varphi}$ ) the Levi-Civita connection  and the curvature tensor  of the Euclidean space $\R^3$ (respectively, of the Ilmanen's space $\Omega^\varphi$). Then, for any  orthonormal frame  $\{e_i\}_{i=1,2,3}$ of $\mathbb{R}^{3}$ we obtain 
\begin{align}
\label{levicivita}
&D^{\varphi}_{X}Y=D_{X}Y+\frac{1}{2}\dot{\varphi}\left(\langle X,e_{3}\rangle Y+\langle Y,e_{3}\rangle X-\langle X,Y\rangle e_{3}\right), 
\end{align}
for any  tangent vectors fields $X,Y\in T\Omega$ and  
\begin{align}
R^{\varphi}(X,Y,Y,X)&=-\frac{e^{\varphi}}{4}\vert X\vert ^2\left( (2\ddot{\varphi}-\dot{\varphi}^{2})\langle Y,e_{3}\rangle^{2}+\vert Y\vert^{2} \dot{\varphi}^2\right)\label{Riemman}\\
& + \frac{e^{\varphi}}{4}\langle X,Y\rangle \left( (2\ddot{\varphi}-\dot{\varphi}^{2})\langle Y,e_{3}\rangle \langle X,e_{3}\rangle+\langle X,Y\rangle \dot{\varphi}^2\right)\nonumber \\
&+ \frac{e^{\varphi}}{4}\langle X,Y\rangle \left( (2\ddot{\varphi}-\dot{\varphi}^{2})\langle Y,e_{3}\rangle \langle X,e_{3}\rangle\right)\nonumber\\
&- \frac{e^{\varphi}}{4}\vert Y\vert^2 \left( (2\ddot{\varphi}-\dot{\varphi}^{2}) \langle X,e_{3}\rangle^2\right)\nonumber
\end{align}
for any tangent vector fields $X,Y\in T\Omega$, where $(\dot{\ \ })$ denotes the derivative with respect to  $x_{3}$.

From \eqref{levicivita} and \eqref{Riemman}, we also have the following 
\begin{lemma}
\label{simbolos}
Consider the  orthonormal frame of $\Omega^\varphi$ given by  $\{e_{i}^{\varphi}=e^{-\frac{\varphi}{2}}e_{i}\}$. Then,
\begin{align*}
&\langle D^{\varphi}_{e_{i}^{\varphi}}e_{j}^{\varphi},e_{k}^{\varphi}\rangle^{\varphi}=\frac{1}{2}e^{-\frac{\varphi}{2}}\dot{\varphi}(\delta_{3 j}\delta_{i k}-\delta_{i j}\delta_{3 k}) \\
&\mathbb{K}^{\varphi}(e_{i}^{\varphi},e_{j}^{\varphi})=\frac{1}{4}e^{-\varphi}\left((\dot{\varphi}^{2}-2\ddot{\varphi})\delta_{i3}-\dot{\varphi}^{2} \right) \text{ for } i\neq j. \\
 \overline{\nabla}^\varphi\mathbb{K}^{\varphi}(e_{i}^{\varphi},&e_{j}^{\varphi})=\frac{1}{4}\left(\dot{\varphi}^{3}-(\dot{\varphi}^{3}-2\dot{\varphi}\ddot{\varphi})\delta_{i3}+2(\dot{\varphi}\ddot{\varphi}-\dddot{\varphi})\delta_{i3}-2\dot{\varphi}\ddot{\varphi} \right)e_{3},
\end{align*}
where $\overline{\nabla}^\varphi$  and   $\mathbb{K}^{\varphi}$  are, respectively,  the usual gradient operator and the sectional curvature of $\Omega^{\varphi}$.\end{lemma}

\begin{definition} {\rm We say that the Ilmanen's space $\Omega^\varphi$ has {\sl bounded geometry} if the sectional curvature $\mathbb{K}^{\varphi}$ is bounded and the injectivity radius is bounded from below.}
\end{definition}
From Lemma \ref{simbolos} and the work of Cheeger, Gromov and Taylor \cite{CGT}, we can prove,
\begin{proposition}
\label{geometriacotada}
The following statements hold
\begin{enumerate}
 \item If $\varphi$ is a positive smooth function outside of a compact set, then the Ilmanen's space is complete.
\item If $\varphi$ is a smooth function such that $e^{-\varphi}\text{max}\{\dot{\varphi}^{2},\ddot{\varphi}\}$ is bounded outside a compact set, then the Ilmanen's space has bounded geometry.
\end{enumerate}
\end{proposition}

\begin{remark}\label{r1}Throughout this paper, we will consider $\Sigma$ as a connected and  orientable surface with empty boundary in $\Omega\subseteq \R^3$.
\end{remark} Denote by $N$  and $\mathcal{S}$  the Gauss map and the second fundamental form of $\Sigma$  in  $\mathbb{R}^{3}$, respectively.  Then  the shape operators $\textbf{S}^{\varphi}$ and $ \textbf{S}$  of $\Sigma$ in $\Omega^\varphi$  and  $\R^3$, respectively,   satisfy
$$-\textbf{S}^{\varphi}_{p} u=D^{\varphi}_{u}N^{\varphi}=e^{-\varphi/2}\left(-\textbf{S}_{p} u+\frac{1}{2}\dot{\varphi}\langle N(p),e_{3}\rangle u \right).$$
for  any point $p\in \Sigma$ and  any vector $u\in T_{p}\Sigma$, where  $N^{\varphi}=e^{-\varphi/2}N$ is the Gauss map of $\Sigma$  in the Ilmanen's space. The above relation gives\begin{proposition}
\label{SfundamentalIlmanen} 

\begin{align*}
 \mathcal{S}^{\varphi}_{p}(u,v)&=e^{\varphi/2}\left(\mathcal{S}_{p}(u,v)+\frac{1}{2}\dot{\varphi}\langle N(p),e_{3}\rangle\langle u,v\rangle \right),\\
k_{i}^{\varphi}(p)&=e^{-\varphi/2}\left(k_{i}(p)+\frac{1}{2}\dot{\varphi}\langle N(p),e_{3}\rangle\right)
\end{align*}
for any $ u,v\in T_{p}\Sigma$, where $\mathcal{S}^{\varphi}$ and $k_{i}^{\varphi}$  (respectively, $\mathcal{S}$ and $k_{i}$) are the second fundamental form and the principal curvatures of $\Sigma$ in the $\Omega^\varphi$ (respectively in $\R^3$).

In particular, the corresponding  mean curvatures satisfy
$$H^{\varphi}=e^{-\varphi/2}\left(H+\dot{\varphi}\langle N,e_{3}\rangle \right).$$
\end{proposition}

\section{Short background of $[\varphi,\vec{e}_{3}]$-minimal surfaces.}

 In what follows, $\nabla$, $\Delta$ and $\nabla^{2}$ will denote, respectively, the gradient, Laplacian and Hessian operators on $\Sigma$ associated to the induced metric  from $\R^3$. 

\begin{definition}
\label{defphimin}
An orientable immersion $\Sigma$ in $\mathbb{R}^{3}$ is called $[\varphi,\vec{e}_{3}]$-minimal if and only if the mean curvature $H$ verifies that
$H=-\langle\overline{\nabla}\varphi,N\rangle,$ 
where $\overline{\nabla}$ is the  gradient in $\mathbb{R}^{3}$.
\end{definition}
A $[\varphi,\vec{e}_{3}]$-minimal $\Sigma$ in $\R^3$   can be   viewed either as a critical point of the weighted area functional
\begin{equation}
\label{critiarea}
{\cal A}^{\varphi}(\Sigma):=\int_{\Sigma}e^{\varphi}\, dA_{\Sigma},
\end{equation}
where $dA_{\Sigma}$ is the volume element of $\Sigma$, or as a  minimal surface in   the Ilmanen's space $\Omega^\varphi$.
From this property of minimality, a tangency principle can be applied and  any two different $\varphi$-minimal  surfaces cannot ``touch" each other at one interior or boundary point (see \cite[Theorem 1 and Theorem 1a]{E}).

Let $\Sigma$ be a $[\varphi,\vec{e}_{3}]$-minimal surface and denote  by
$$\mu:=\langle p,\vec{e}_{3}\rangle, \ \ \eta:=\langle N(p),\vec{e}_{3}\rangle, \ \ p\in\Sigma,$$
their  height and angle function, respectively. The following list of fundamental equations that will appear throughout this paper were proved in the Section 2 of \cite{MM}.
\begin{lemma}
\label{laplacianalturangulo} Then, the following relations hold,
\begin{enumerate}
\item $\nabla\mu=\vec{e}_3^{\top}, \ \ \langle\nabla\eta,\cdot\rangle=\mathcal{S}(\nabla\mu,\cdot),$
\item $\dot{\varphi}^{2}=\dot{\varphi}^{2}\vert\nabla\mu\vert^{2}+H^{2},$
\item $\dot{\varphi}\nabla^{2}\mu=H\mathcal{S},$
\item $\nabla^{2}\eta=\nabla_{\nabla\mu}\mathcal{S}-\eta\mathcal{S}^{[2]}, $
\item $\Delta\mu=\dot{\varphi}(1-\vert\nabla\mu\vert^{2}),$
\item $\Delta N+\dot{\varphi}\nabla\eta+\ddot{\varphi}\eta\nabla\mu+\vert\mathcal{S}\vert^{2}N=0,$
\item $\nabla^{2}H=-\eta\nabla^{2}\dot{\varphi}-\nabla_{\nabla\varphi}\mathcal{S}-H\mathcal{S}^{[2]}+\mathcal{B},$
\item $\Delta\mathcal{S}+\nabla_{\nabla\varphi}\mathcal{S}+\eta\nabla^{2}\dot{\varphi}+\vert\mathcal{S}\vert^{2}\mathcal{S}-\mathcal{B}=0, $
\end{enumerate}
where $\vec{e}_3^{\top}$ denotes the tangent projection of $\vec{e}_{3}$ in $T\Sigma$, $\nabla_X$ is the Levi-Civita connection induced by $D$ and $\mathcal{S}^{[2]}$ and $\mathcal{B}$ are the symmetric 2-tensors given by the following expressions:
\begin{align*}
&\mathcal{S}^{[2]}(X,Y)=\sum_{k}\mathcal{S}(X,E_{k})\mathcal{S}(E_{k},Y) \\
&\mathcal{B}(X,Y)=\langle\nabla\dot{\varphi},X\rangle\mathcal{S}(\nabla\mu,Y)+\langle\nabla\dot{\varphi},Y\rangle\mathcal{S}(\nabla\mu,X)
\end{align*}
for any vector fields $X,Y\in T\Sigma$ and any orthornomal frame $\{E_{i}\}$ of $T\Sigma$.
\end{lemma}

From the strong maximum principle, the equation 6 in Lemma \ref{laplacianalturangulo} and Definition \ref{defphimin} the following result holds:
\begin{theorem}
\label{tanprinH}
Let $\varphi:]a,b[\rightarrow\mathbb{R}$ be a strictly increasing function satisfying
\begin{equation}
\label{condiciones}
\ddot{\varphi}+\lambda\dot{\varphi}^{2}\geq 0, \emph { for some } \lambda>0,
\end{equation}
and let $\Sigma$ be a $[\varphi,\vec{e}_{3}]$-minimal immersion in $\mathbb{R}^{2}\times]a,b[$ with $H\leq 0$. If $H$ vanishes anywhere, then $H$ vanishes everywhere and $\Sigma$ lies in a vertical plane.
\end{theorem}

Using the  Hamilton's principle (see \cite[Section 2]{SAVAS}) we also can prove,
\begin{theorem}
\label{tanprinK}
Let $\varphi:]a,b[\rightarrow\mathbb{R}$ be a strictly increasing function satisfying $\dddot{\varphi}\leq 0$ and let $\Sigma$ be a locally convex $[\varphi,\vec{e}_{3}]$-minimal immersion in $\mathbb{R}^{2}\times ]a,b[$. If the Gauss curvature $K$ vanishes anywhere, then $K$ vanishes everywhere.
\end{theorem}

\section{Stability of $[\varphi,\vec{e}_{3}]$-minimal surfaces.}

In this section, we will study the stability of $[\varphi,\vec{e}_{3}]$-minimal surfaces, where \textit{stable} means stability as minimal surface in the Ilmanen's space, i.e., its   weighted area functional $\mathcal{A}^{\varphi}$ is locally minimal. 

\begin{proposition}{{\rm (see  \cite[Appendix]{CMZ})}}
\label{segundavar}
Let $X$ be a compactly supported variational vector field on the normal bundle of $\Sigma$ and $F_{t}$  the normal variation associated to $X$. If $\Sigma$ is an oriented $[\varphi,\vec{e}_{3}]$-minimal surface, then the second derivative of the weighted area functional $\mathcal{A}^{\varphi}$ is given by,
$$\frac{d^{2}}{dt^{2}}\bigg\vert_{t=0}\mathcal{A}^{\varphi}(F_{t}(\Sigma))=\mathcal{Q}_{\varphi}(u,u), \emph{ for any } u\in\mathcal{C}^{\infty}_{0}(\Sigma).$$
where $\mathcal{Q}_{\varphi}$ is the symmetric bilinear 
\begin{equation}
\label{bilinearform}
\mathcal{Q}_{\varphi}(f,g)=\int_{\Sigma}e^{\varphi}\left(\langle\nabla f,\nabla g\rangle-(\vert \mathcal{S}\vert^2-\ddot{\varphi}\eta^{2})fg\right)\, d\Sigma.
\end{equation}
\end{proposition}

\begin{definition}
We say that an oriented $[\varphi,\vec{e}_{3}]$-minimal surface $\Sigma$ without boundary is stable if and only if for any compactly supported smooth function $u$, it holds that
\begin{equation}
\label{phistable1}
\mathcal{Q}_{\varphi}(u,u)=-\int_{\Sigma}u\mathcal{L}_{\varphi}(u)\, e^{\varphi}\, d\Sigma\geq 0,
\end{equation}
where $\mathcal{L}_{\varphi}$  is a gradient Schr\"ondinger operator defined on $\mathcal{C}^{2}(\Sigma)$ by
\begin{equation}
\label{operatorstability}
\mathcal{L}_{\varphi}(\cdot):=\Delta^{\varphi} (\cdot)+(\vert\mathcal{S}\vert^{2}-\ddot{\varphi}\eta^{2})(\cdot)
\end{equation}
and $\Delta^{\varphi}$ is the drift Laplacian given by $\Delta^{\varphi}(\cdot)=\Delta(\cdot)+\langle\nabla\varphi,\nabla(\cdot)\rangle.$
\end{definition}

\begin{remark}
\label{noexistencia}
{\rm The existence of stable surfaces it is not guaranteed for any function $\varphi$. Cheng,  Mejia and Zhou \cite{CMZ}  proved that if $\Omega^\varphi$ is  complete and  $\ddot{\varphi}\leq -\varepsilon<0$ for some positive constant $\varepsilon$, then there are not stable surfaces without boundary and with finite weighted area.}
\end{remark}

\begin{proposition}
\label{stableHpositiva}
Let $\varphi:]\alpha,+\infty[ \longrightarrow \R$ be a regular function satisfying \eqref{c1}  and $\Sigma$ be an oriented $[\varphi,\vec{e}_{3}]$-minimal immersion in $\mathbb{R}^{3}$ with $H\leq 0$. Then, $\Sigma$ is stable.
\end{proposition}
\begin{proof}
From Theorem \ref{tanprinH}, we can assume that $H<0$ everywhere otherwhise $\Sigma$ is a vertical plane and as we are going to see in Corollary \ref{grafosestables}, $\Sigma$ will be stable. 

Suppose $H<0$ and  consider $w=\log(\eta)$, then  by Equation 6 of Lemma \ref{laplacianalturangulo}, we get that
\begin{equation}
\label{lapw}
\Delta w+\langle\nabla\varphi,\nabla w\rangle=-\frac{\vert\nabla\eta\vert^{2}}{\eta^{2}}-\vert\mathcal{S}\vert^{2}-\ddot{\varphi}\vert\nabla\mu\vert^{2}.
\end{equation}
Now, fix any compact domain $\mathcal{K}$ on $\Sigma$ and consider $u$ as an arbitrary function $\mathcal{C}^{2}(\Sigma)$ with compact support inside $\mathcal{K}$. Applying the divergence theorem to the expression $\text{div}\left(e^{\varphi}\, u^{2}\, \nabla w \right)$ we have,
\begin{equation}
\label{divw}
\int_{\Sigma}e^{\varphi}\, u^{2}\, \left(\Delta w+\langle\nabla\varphi,\nabla w\rangle\right)\, d\Sigma=-2\int_{\Sigma}e^{\varphi}\, u\, \langle\nabla u,\nabla w\rangle\, d\Sigma.
\end{equation}
Now, from \eqref{lapw}, \eqref{divw} and \eqref{bilinearform} we obtain,
$$\mathcal{Q}(u,u)=\int_{\Sigma}e^{\varphi}\left(\vert\nabla u-\frac{u}{\eta}\nabla\eta\vert^{2}+\ddot{\varphi}u^{2} \right)\, d\Sigma\geq 0.$$
which concludes the proof.
\end{proof}

Fischer-Colbrie and Schoen \cite{FCS} gave a condition on the first eigenvalue $\lambda_{1}(\mathcal{L}_{\varphi})$ of $\mathcal{L}_{\varphi}$ which  characterizes the stability of minimal surfaces in 3-manifolds. Using this characterization  we have,
\begin{proposition}
\label{FischerColbrie}
Let $\Sigma$ be a complete oriented $[\varphi,\vec{e}_{3}]$-immersion in $\mathbb{R}^{3}$. The following statements are equivalent
\begin{enumerate}
\item $\Sigma$ is stable.
\item The first eigenvalue $\lambda_{1}(\mathcal{L}_{\varphi})(\mathcal{K})<0$ on any compact domain $\mathcal{K}\subset\Sigma$.
\item There exists a positive function $u\in\mathcal{C}^{2}(\Sigma)$ such that $\mathcal{L}_{\varphi}(u)=0$.
\end{enumerate}
\end{proposition}

As consequence of Proposition \ref{FischerColbrie}, we have the following corollary:

\begin{corollary}
\label{grafosestables}
Let $\Sigma$ be a complete oriented $[\varphi,\vec{e}_{3}]$-minimal surface in $\mathbb{R}^{3}$. Then,
\begin{enumerate}
\item If $\Sigma$ is a graph with respect to a Killing vector $V$ lying in the orthogonal complement of $\vec{e}_{3}$, then $\Sigma$ is stable for any smooth function $\varphi$.
\item If $\varphi$ is an increasing convex smooth function and $\Sigma$ is a graph with respect to $\vec{e}_{3}$, then $\Sigma$ is stable. 
\end{enumerate}
\end{corollary}
\begin{proof}
Consider the following smooth function $\nu=\langle V,N\rangle$. By the assumption, $\nu$  is a positive function on $\Sigma$ and from Equation 6 of Lemma \ref{laplacianalturangulo}, we get that
\begin{equation}
\label{lapu}
\Delta \nu=-\dot{\varphi}\langle V,\nabla\eta\rangle-\ddot{\varphi}\eta\langle V,\nabla\mu\rangle-\vert\mathcal{S}\vert^{2}\nu.
\end{equation}
On the other hand, by Equation 1  in Lemma \ref{laplacianalturangulo}, the following relations hold,
\begin{align*}
&\langle\nabla\varphi,\nabla \nu\rangle=-\dot{\varphi}\langle\textbf{S}(V,\nabla\mu),N\rangle=\dot{\varphi}\mathcal{S}(\nabla\mu,N)=\dot{\varphi}\langle\nabla\eta, V\rangle, \\
&\langle V,\nabla\mu\rangle=\langle V,\vec{e}_{3}-\eta N\rangle=-\eta \nu.
\end{align*}
From the  above expressions and \eqref{lapu}, we have  $\mathcal{L}_{\varphi}(u)=0$ and the first statement holds. The second assertion is a consequence of Proposition \ref{stableHpositiva}.
\end{proof}

\begin{remark}
{\rm  Some results about stable $[\varphi,\vec{e}_{3}]$-minimal surface with $\ddot{\varphi}<0$ can be found  in \cite{MS}. }
\end{remark}

Finally, from Theorem 3 in \cite{CMZ} and Corollary \ref{grafosestables}, we also can prove the following non-existence result:
\begin{theorem}
 Let $V$ be a Killing vector field in the orthogonal complement of $\vec{e}_{3}$. If $\varphi$ a smooth function such that $\ddot{\varphi}\leq -\varepsilon<0$ ,for some $\varepsilon>0$, and the Ilmanen's space is  complete, then there are not $[\varphi,\vec{e}_{3}]$-minimal graphs respect to $V$ with finite weighted area.
\end{theorem}

\subsection{Intrinsic area estimates}
To prove intrinsic area bounds we will follow the same method as in \cite{SX}.

\

Let $\varphi:\R\longrightarrow \R$  be a smooth function satisfying \eqref{c1} and \eqref{c2} and  $\Sigma$ be a  $[\varphi,\vec{e}_{3}]$- immersion in $\mathbb{R}^{3}_\alpha$ with $H\leq 0$.  Consider $\mathcal{D}_{\rho}(p)$ an intrinsic ball in $\Sigma$ of radius $\rho$ centered at  $p$.
\begin{lemma}
\label{firstlemma}
 If $\rho\dot{\varphi}(\rho+\mu(p))<\sqrt{2}\pi$, then $\mathcal{D}_{\rho}(p)$ is disjoint from the conjugate locus of $p$ and
\begin{equation}
\label{ine1}
\int_{\Sigma}\vert\mathcal{S}\vert^{2}\, u^{2}\, d\Sigma\leq e^{2\rho\dot{\varphi}(\rho+\mu(p))}\int_{\Sigma}(\vert\nabla u\vert^{2}+\ddot{\varphi}\eta^2u^{2})\, d\Sigma,
\end{equation}
 for any $u\in\mathcal{H}^{2}_{0}(\mathcal{D}_{\rho}(p)).$
\end{lemma}
\begin{proof}
As $\vert\nabla\mu\vert^{2}\leq 1$, it is clear that for any $q\in\mathcal{D}_{\rho}(p)$ , $\mu(p)-\rho\leq\mu(q)\leq \mu(p)+\rho .$ Hence $$\varphi(\mu(p)-\rho)\leq\varphi(\mu(q))\leq\varphi(\rho+\mu(p)), \ \ q\in\mathcal{D}_{\rho}(p)$$ 
and we have the following control of the curvature
$$2K\leq H^{2}\leq\dot{\varphi}^{2}(\mu(q))\leq\dot{\varphi}^{2}(\rho+\mu(p)) \text{ on }\mathcal{D}_{\rho}(p).$$

Consequently, the first statement follows from the Rauch comparison Theorem. Finally, the inequality \eqref{ine1} follows from the above inequalities, Proposition \ref{stableHpositiva} and the stability inequality \eqref{phistable1}.
\end{proof}

\begin{theorem}[Boundness of area]
\label{boundnessarea}
Let  $\Sigma$ be a  $[\varphi,\vec{e}_{3}]$- immersion in $\mathbb{R}^{3}_\alpha$ with $H\leq 0$ and $\varphi$ satisfying \eqref{c1} and \eqref{c2}. If  $2\rho\dot{\varphi}(\rho+\mu(p))<\log(2)$ and $ \sqrt{\vert \Gamma \vert} \ \rho < 1$, then the geodesic disk $\mathcal{D}_{\rho}(p)$ of radius $\rho$  centered at $p$ is disjoint from the cut locus of $p$ and 
\begin{equation}
\label{cotarea}
\mathcal{A}(\mathcal{D}_{\rho}(p))< 4\pi\rho^{2},
\end{equation}
 where $\mathcal{A}(\cdot)$ is the intrinsic area of $\Sigma$ in $\R^3$.
\end{theorem}

\begin{proof}
First, we prove the inequality \eqref{cotarea}. Since $\vert\mathcal{S}\vert^{2}=H^{2}-2K$, from \eqref{c1}, \eqref{c2} and Lemma \ref{firstlemma}, we get that for any $u\in\mathcal{H}^{2}_{0}(\mathcal{D}_{\rho}(p))$
\begin{align}
-2\int_{\Sigma} K\, u^{2}\, d\Sigma&\leq e^{2\rho\dot{\varphi}(\rho+\mu(p))}\int_{\Sigma}(\vert\nabla u\vert^{2}+\ddot{\varphi}\eta^2u^{2})\, d\Sigma -  \int_{\Sigma}\, \dot{\varphi}^2 \eta^2u^{2}\, d\Sigma \label{ba1}\\&
\leq 2\int_{\Sigma}\vert\nabla u\vert^{2}\, d\Sigma+\Gamma \int_{\Sigma}\, \eta^2u^{2}\, d\Sigma. \nonumber
\end{align}

Moreover, by Gauss-Bonnet, the variation of the length $l(s)$ of $\partial\mathcal{D}_{s}(\rho)$ is given by,
\begin{equation}
\label{variationl}
l'(s)=\int_{\partial\mathcal{D}_{s}(p)}k_{g}\, d\sigma=2\pi-\int_{\mathcal{D}_{s}(p)}K\, d\Sigma=2\pi-K(s),
\end{equation}
where $k_{g}$ is the geodesic curvature of $\partial\mathcal{D}_{s}(p)$. If $u$ is a radial function  satisfying $u'\leq 0$ and $u(\rho)=0$, the coarea formula gives,
\begin{align*}
&\int_{\mathcal{D}_{s}(p)}K\, u^{2}\, d\Sigma=\int_{0}^{\rho}u^{2}(s)\int_{\partial\mathcal{D}_{s}(p)}K\, d\sigma\, ds=\int_{0}^{\rho}u^{2}(s)K'(s)\, ds, \\
&\int_{\mathcal{D}_{s}(p)}\vert\nabla u\vert^{2} d\Sigma=\int_{0}^{\rho}\int_{\partial\mathcal{D}_{s}(p)}\vert\nabla u\vert^{2} d\sigma\, ds=\int_{0}^{\rho}(u'(s))^{2}l(s)\, ds.
\end{align*}
In particular, by taking $u(s)=1-\frac{s}{\rho}$, applying integration by parts and using \eqref{variationl} and the above expressions, we have
\begin{align}
\label{ineq}
-4\pi+4\frac{\mathcal{A}(\mathcal{D}_{\rho}(p))}{\rho^{2}}&= -\frac{4}{\rho}\int_0^\rho (2 \pi - l'(s))( 1-\frac{s}{\rho}) ds \\
& \leq 2\frac{\mathcal{A}(\mathcal{D}_{\rho}(p))}{\rho^{2}}+\Gamma \int_{\Sigma}\, \eta^2u^{2}\, d\Sigma. \nonumber
\end{align}
If $\Gamma\leq 0$, then the inequality \eqref{cotarea} trivially holds. If $\Gamma > 0$, using that $\sqrt{\Gamma} \rho <1$ and \eqref{ineq} we get 
$$\mathcal{A}(\mathcal{D}_\rho(p))\leq\frac{4\pi}{2-\Gamma\rho^{2}}\rho^{2}< 4\pi\rho^{2}.$$

Now we will see  that $\mathcal{D}_{\rho}(p)$ is disjoint from the cut locus of $p$. Otherwise,  there exists $q\in\partial\mathcal{D}_{r_{0}}(p)$ that lies in the cut locus of $p$ where $r_{0}=\text{Inj}(\Sigma)(p)\leq\rho$. Since $ \rho\dot{\varphi}(\rho+\mu(p))<\sqrt{2}\pi$, from Lemma \ref{firstlemma} and  a Klingenberg-type argument (see for example \cite[Chapter 5]{PP}), there exist two geodesics from $p$ to $q$ which bound a smooth  domain $D \subset\mathcal{D}_{r_{0}}(p)$ with a possible corner at $p$. By the Gauss-Bonnet,
$$2\pi= 2\pi-\int_{\partial D}k_{g}\, d\sigma=\int_{D} K\, d\Sigma\leq\frac{1}{2}\dot{\varphi}^{2}(r_{0}+\mu(p))\mathcal{A}(D).$$
Hence, 
$$\mathcal{A}(\mathcal{D}_{r_{0}}(p))\geq\mathcal{A}( D)\geq 4\pi/\dot{\varphi}^{2}(r_{0}+\mu(p)).$$ From the area estimate \eqref{cotarea} for $\rho=r_{0}$ and the fact that $\rho\dot{\varphi}(\rho+\mu(p))<\log(2)/2$, we get that
$$4\pi>4\pi r_{0}^{2}\dot{\varphi}^{2}(r_{0}+\mu(p))\geq4\pi,$$
which is a contradiction.
\end{proof}

\subsection{Blow-up  and curvature estimate.}
For later use we will need the following compactness result which is a consequence of Theorem 2.1 in    \cite{W2}:
\begin{theorem}
\label{CompacWhite}
Let $\Omega$ be an open subset of $\mathbb{R}^{3}$. Let $\{\varphi_{n}\}$ be a sequence of smooth functions on $\Omega$ converging smoothly to  $\varphi_{\infty}$. Let $\Sigma_{n}$ be a sequence of properly embedded minimal surfaces in the corresponding  Ilmanen's space $\Omega^{\varphi_{n}}$. Suppose also that the area and the genus of $\Sigma_{n}$ are bounded uniformly on compact subsets of $\Omega$. Then, the total curvatures of $\Sigma_{n}$ are also uniformly bounded on compact subsets of $\Omega$ and after passing to a subsequence, $\Sigma_{n}$ converge to a smooth properly embedded minimal $\Sigma_{\infty}$ in $\Omega^{\varphi_\infty}$. The convergence is smooth away from a discrete set $\mathcal{C}$ and for each connected component $\Sigma^0_\infty$ of $\Sigma_\infty$ either,
\begin{enumerate}
\item the convergence  to $\Sigma^0_\infty$ is smooth everywhere with multiplicity $1$, or
\item the convergence  to $\Sigma^0_\infty$  is smooth with some multiplicity grater than $1$ away from $\Sigma_\infty\cap\mathcal{C}$. In this case, if $\Sigma_\infty$ is two-sided, the it must be stable.
\end{enumerate}
If the total curvatures of $\Sigma_{n}$ are bounded by $\beta$, the set $\mathcal{C}$ has at most $\beta/4\pi$ points.
\end{theorem}

Following the same method as  in \cite{W1}, we  prove 

\begin{lemma}[Monoticity formula]
\label{monoticity}
Let  $\Sigma$ be a $[\varphi,\vec{e}_{3}]$-minimal immersion in $\R^3_\alpha$ with $\varphi$ satisfying  \eqref{c1}. Fix any point $q\in\Sigma$ and consider $B(q,r)$  the Euclidean ball of radius $r$ centered at $q$. Denote by $\Sigma_{r}=\Sigma\cap B(q,r)$ and by  $\partial\Sigma_{r}=\Sigma\cap\partial B(q,r)$ and define $A(r)=\mathcal{A}(\Sigma\cap B(q,r))$ and $L(r)=\text{{\rm length}}(\Sigma\cap\partial B(q,r))$. If there exists $\varepsilon>0$ such that $0\leq\varphi(\varepsilon)<1$, then the function $$\mathcal{O}_{\Sigma}(r)=\frac{\varphi(r) A(r)}{4\pi r^{2}}$$
is increasing in $r$ over the interval $]0,\varepsilon[$.
\end{lemma}
\begin{proof}
If we take on $\Sigma$ the  vector field $X(p)=p-q$, $p\in \Sigma$, then, from the Divergence Theorem, we get that
\begin{align}
\label{e1}
2 A(r)&=\int_{\Sigma_{r}}\text{div}(X)\, d\Sigma_r=\int_{\partial\Sigma_{r}}\langle X,\nu\rangle\, d\sigma-\int_{\Sigma_{r}}H \langle X,N\rangle \, d\Sigma_r\\
& = \int_{\partial\Sigma_{r}}\langle X,\nu\rangle\, d\sigma_r + \int_{\Sigma_{r}}\dot{\varphi} \, \eta \, \langle X,N\rangle \, d\Sigma_r \nonumber
\end{align}
where $\nu$ is the conormal vector over $\partial\Sigma_{r}$, $d\Sigma_r$ is the volume element of $\Sigma$ induced by the Euclidean metric and $d\sigma$ is the length element of $\partial\Sigma_{r}$. From hypothesis, we have that  $0\leq\varphi(r)\leq 1$ for any $r<\varepsilon$. Moreover, as in the proof of Theorem 3 in \cite{W1}, $L(r)\leq A'(r)$ for any $r$ and  joining both inequalities to the expression \eqref{e1}, we have
\begin{equation}
\label{e2}
0\leq r A'(r)+r\dot{\varphi}(r)A(r)-2A(r).
\end{equation}
Finally, multiplying by $r^{-3}\varphi(r)$ in \eqref{e2}, we get
\begin{align}
0&\leq r^{-2}\varphi(r)A'(r)+r^{-2}\dot{\varphi}(r)\varphi(r)A(r)-2r^{-3}\varphi(r)A(r) \\ 
&\leq r^{-2}\varphi(r)A'(r)+r^{-2}\dot{\varphi}(r)A(r)-2r^{-3}\varphi(r)A(r)=(r^{-2}\varphi(r)A(r))'.\nonumber
\end{align}
which concludes the proof.
\end{proof}

\begin{theorem}[Blow-up]
\label{blowup}
Let   $\Sigma$ be a properly embedded $[\varphi,\vec{e}_{3}]$-minimal surface in $\mathbb{R}^{3}_\alpha$ with $H\leq 0$,  locally bounded genus and $\varphi$ satisfying \eqref{c1} and \eqref{cc3}. Consider any sequence $\{\lambda_{n}\}\rightarrow+\infty$ and suppose that there exists a sequence $\{p_{n}\}$ in $\Sigma$ such that  $\{\dot{\varphi}(\mu(p_{n}))/\lambda_{n}\}\rightarrow C$ for some constant $C\geq 0$ .  Then, after passing to a subsequence,   $\Sigma_n=\lambda_n ( \Sigma-p_n)$  converge smoothly to\begin{enumerate}[(i)]
\item   a plane  when $C=0$,
\item one of the following translating soliton when $C>0$:
\begin{enumerate}[(a)]
\item vertical plane,
\item grim reaper surface,
\item titled grim reaper surface,
\item bowl soliton,
\item $\Delta$-Wing translating soliton.
\end{enumerate}
\end{enumerate}
\end{theorem}
\begin{proof}
 Consider the sequence of properly embedded surfaces $\Sigma_{n}=\lambda_{n}(\Sigma-p_{n})$ in $\mathbb{R}^{3}$. Each  $\Sigma_n$ is a minimal surface in the Ilmanen's space $\Omega^{\varphi_{n}}$ where $\Omega= \R^3$ and 
\begin{equation*}\varphi_{n}(x_3)=\varphi\left( \frac{x_3}{\lambda_{n}}+\mu(p_{n})\right)-\varphi(\mu(p_{n})).\end{equation*}
It is clear form our assumption that 
\begin{equation}
\varphi_n\rightarrow \varphi_\infty, \quad \textnormal{with} \quad \varphi_\infty(x_3) =  C \, x_3\label{varphin}.
\end{equation}
For any compact set $\mathcal{K}$ in $\Omega$, we can consider $r$ large enough such that $\mathcal{K}$ is contained in the  Euclidean ball $B(0,r)$ of radius $r$ centered at the origin. Then, for any $\epsilon_0>0$ and  $n$ large enough, it follows from \eqref{varphin} that

\begin{align*}
\mathcal{A}^{\varphi_{n}}( \Sigma_{n}\cap \mathcal{K})&\leq\int_{\Sigma_{n}\cap B(0,r)}e^{\varphi_{n}}\, d\Sigma_{n}=\int_{\Sigma_{n}\cap B(0,r)}e^{C q + \epsilon_0}\, d\Sigma_{n}\\ &\leq 
 \lambda_{n}^{2}\int_{\Sigma\cap B(p_{n},\frac{r}{\lambda_{n}})}e^{C r + \epsilon_0}\, d\Sigma = e^{C r + \epsilon_0}\lambda_{n}^{2}\mathcal{A}(\Sigma\cap B(p_{n},r/\lambda_{n})).
\end{align*}

As $\varphi$ can  be choose up to a constant, we can assume that there exists $\varepsilon>0$ such that $0<\varphi(\varepsilon)<1$. Since $r/\lambda_{n}\rightarrow 0$, it follows Lemma \ref{monoticity} that there must be $n_{0}$  such that $r/\lambda_{n}\leq \varepsilon$ and $\mathcal{O}_{\Sigma}( r/\lambda_{n})\leq\mathcal{O}_{\Sigma}(\varepsilon)$ for any $n\geq n_{0}$. Thus, 
\begin{align*}
\mathcal{A}(\Sigma\cap B(p_{n},r/\lambda_{n}))\leq\frac{\varphi(\varepsilon)}{\varphi( r/\lambda_{n})}\left(\frac{r}{\lambda_{n}} \right)^{2}\frac{\mathcal{A}(\Sigma\cap B(p_{n},\varepsilon))}{\varepsilon^{2}}.
\end{align*}
Joining both inequalities we have that, for $n$ large enough, 
$$\mathcal{A}^{\varphi_{n}} (\Sigma_{n}\cap\mathcal{K})\leq \frac{e^{C r + \epsilon_0} \varphi(\varepsilon)}{\varphi( r/\lambda_{n})}r^{2}\frac{\mathcal{A}(\Sigma\cap B(p_{n},\varepsilon))}{\varepsilon^{2}}\leq 4\pi \frac{e^{C r + \epsilon_0}\varphi(\varepsilon)}{\varphi( r/\lambda_{n})}r^{2} $$
As $\lambda_{n}\rightarrow +\infty$, there exists a positive constant $\Theta$, depending only of $\varphi$ , such that
$$\mathcal{A}^{\varphi_{n}}(\Sigma_{n}\cap \mathcal{K})\leq \Theta \, \pi e^{Cr} r^{2}.$$

Consequently, $\Sigma_n$ have area uniformly bounded on compact subsets of $\R^3$. From Theorem \ref{CompacWhite}, $\Sigma_{n}$ converge  to a properly embedded $[\varphi_{\infty},\vec{e}_{3}]$-minimal surface $\Sigma_{\infty}$ in $\mathbb{R}^{3}$. Since each $\Sigma_{n}$ is stable, $\Sigma_{\infty}$ must be a plane in $\mathbb{R}^{3}$ if $C=0$ (see \cite{FCS}). If $C>0$, then $\Sigma_{\infty}$ is a mean convex properly embedded  translating soliton in $\mathbb{R}^{3}$ and  from  the results in \cite{HIMW2} and \cite{SX}, $\Sigma_{\infty}$ must be either a vertical plane, a grim reaper surface, a titled grim greaper surface, a bowl soliton or a $\Delta$-Wing translating soliton. 

Finally, if $p_n\in \Sigma_n$ converge to $p\in \Sigma_\infty$ and the length of the second fundamental form of $\Sigma_n$ at $p_n$ are such that $\vert{\cal S}_n(p_n)\vert \rightarrow +\infty$, then from the stability of $\Sigma_n$ and Theorem 2.2 in \cite{W2}, if we set $\lambda_n = \vert{\cal S}_n(p_n)\vert$,  we conclude that $\Sigma'_n=\lambda_n(\Sigma_n-p_n)$  converge smoothly with multiplicity 1 to a plane. But this is a contradiction since the length of the second fundamental of $\Sigma_n'$ at the origin satisfy $\vert{\cal S}'_n(0)\vert \rightarrow 1$. In particular the convergence of $\Sigma_n$ is  smooth.
\end{proof}

Now, by combining the methods of Rosenberg, Souam and Toubiana \cite{RST}, and Spruck and Xiao \cite{SX}, we will prove the Theorem A: 
\subsection*{{\sc Proof of Theorem A}}
Suppose that there exists a sequence of points $\{ p_{n}\}$ in $\Sigma$ such that $$\lambda_n=\vert\mathcal{S}\vert (p_n)\rightarrow +\infty,\qquad \lim_{n\rightarrow+\infty} \frac{\lambda_n}{\dot{\varphi}(\mu(p_n))}=+\infty.$$
Then, for  a subsequence  of $\{ p_{n}\}$ we have 
 $ \frac{\dot{\varphi}(\mu(p_n))}{\lambda_n}\mapsto 0$ 
and  from Theorem \ref{blowup}, the sequence  $\Sigma_{n}=\lambda_{n}(\Sigma-p_{n})$ converges smoothly to a plane $\Sigma_{\infty}$ in $\mathbb{R}^{3}$. Since, $\vert\mathcal{S}_{n}(p_{n})\vert=1$ for each $n$ we also have,  $\vert\mathcal{S}_{\infty}(0)\vert=1$, which is a  contradiction. 
{ \hfill $\square$}

\

The following results are consequences of Lemma \ref{simbolos} and the results in \cite{RST,W1}.

\begin{theorem}
\label{segformacoilm1}
Let $\varphi$ a smooth function such that $$\frac{1}{2}e^{-\varphi}\left(\vert\text{max}\{ \dot{\varphi}^{2},\ddot{\varphi}\}\vert\right)+\vert\text{max}\{\dot{\varphi}^{3} ,2\dot{\varphi}\ddot{\varphi} , \dddot{\varphi}\} \vert \geq \rho,$$
for some constant $\rho>0$ and let $\Sigma$ be a minimal surface (possible with boundary) in the Ilmanen's space with total absolute curvature is at most $\lambda<4\pi$. Then there exists a contant $C$ depending of $\lambda$ such that
$$\vert\mathcal{S}^{\varphi}\vert \text{min}\{d_{\varphi}(p,\partial\Sigma),\mathcal{R} \}\leq C \text{ for any } p\in\Sigma,$$
where $$\mathcal{R}=(\text{sup}\vert\mathbb{K}^{\varphi}\vert+\text{sup}\vert\overline{\nabla}^\varphi\mathbb{K}^{\varphi}\vert^{1/2})^{-1}.$$
\end{theorem}

\begin{theorem}
\label{segformacoilm2}
 Let $\varphi$ a smooth function such that the Ilmanen's space is a complete Riemannian manifold with bounded geometry whose sectional curvature $\vert\mathbb{K}^{\varphi}\vert\leq A$ for some constant $A>0$. For any stable minimal immersion $\Sigma$ in the Ilmanen's space (with possible boundary), there exists a constant $C$ such that
$$\vert\mathcal{S}^\varphi\vert \text{min}\{d_{\varphi}(p,\partial\Sigma),\pi/2\sqrt{A} \}\leq C.$$
\end{theorem}

\section{A Spruck-Xiao's type Theorem.}
Using a delicate maximum principle argument, Spruck and Xiao \cite{SX} proved that any complete translating soliton in $\R^3$ with $H\leq 0$  is convex. A slightly simplified proof of this result is presented by Hoffman, Ilmanen, Mart\'in and White \cite{HIMW}. In this section, we consider the same problem for properly embedded $[\varphi,\vec{e}_{3}]$-minimal surfaces in $\R^3$  with $\varphi:\R\longrightarrow \R$ satisfying \eqref{c1}, \eqref{cc3} and $\dddot{\varphi}\leq 0$ on $]\alpha,+\infty[$.

\

We start with some results we will use:
\begin{theorem}{{\rm Generalized Omori-Yau maximum principle for $\Delta^\psi$ \cite[Theorem 3.2]{AMR}}} 
\label{oy}
Let $\Sigma$  be a  surface in $\R^3$ and $\Delta^\psi$ the drift laplacian operator associated to  $\psi\in {\cal C}^2(\Sigma)$. Let $\gamma \in {\cal C}^2(\Sigma)$ be such that
\begin{align}
&\gamma(p) \rightarrow +\infty \quad \text{ as $p\rightarrow \infty$}\\
&\Delta^\psi\gamma \leq C \quad \text{ outside a compact subset of $\Sigma$} \\
&\vert \nabla \gamma \vert \leq C \quad \text{ outside a compact subset of $\Sigma$}
\end{align}
for some constant $C>0$. If $\nu \in {\cal C}^2(\Sigma)$  and $\nu^\star=\sup_\Sigma \nu < +\infty$, then there exists a sequence of points $\{p_n\} \subset \Sigma$ satisfying
\begin{equation}
(i) \ \nu (p_n) > \nu^\star - \frac{1}{n}, \quad (ii)\  \Delta^\psi\nu(p_n) < \frac{1}{n}, \quad (iii)\  \vert \nabla \nu (p_n)\vert < \frac{1}{n},
\end{equation}
for each $n\in \mathbb{N}$.
\end{theorem}
\begin{lemma}
\label{paso 1}
Let $k_{i}$ be the principal curvatures of an immersion $\Sigma$ in $\mathbb{R}^{3}$ and $\mathcal{U}$ the set of totally umbilical points of $\Sigma$. If $\{ v_{i}\}$ is  an orthonormal frame of principal directions in $T\Sigma$, then the following statements hold,
\begin{enumerate}
\item $\nabla_{v_{i}}v_{i}=\alpha_{i}v_{j}, \ \ \nabla_{v_{j}}v_{i}=\alpha_{j}v_{j} \text{ with } \alpha_{i}=-\alpha_{j}$.
\item The coefficients $\alpha_{i}$ are determinated by the formula,
$$\alpha_{i}=\frac{h_{12,i}}{k_{1}-k_{2}} \text{ in } \Sigma-\mathcal{U}, \text{ where } h_{ij,k}=(\nabla_{v_{k}}\mathcal{S})(v_{i},v_{j}).$$
\end{enumerate}
\end{lemma}
\begin{proof}
The first item is trivially obtained by differentiating $\langle v_{i},v_{j}\rangle=\delta_{ij}$  On the other hand, differentiating  $\mathcal{S}(v_{1},v_{2})=0$ and using the first item we get that
$$0=(\nabla_{v_{i}}\mathcal{S})(v_{1},v_{2})+\mathcal{S}(\nabla_{i}v_{1},v_{2})+\mathcal{S}(\nabla_{v_{i}}v_{2},v_{1})=h_{12,i}+\alpha_{i}(k_{2}-k_{1}).$$\end{proof}

\begin{lemma}
\label{paso 2}
If $\Sigma$ is a $[\varphi,\vec{e}_{3}]$-minimal immersion in $\mathbb{R}^{3}$, then
$$\Delta^{\varphi}k_{i}=-\vert\mathcal{S}\vert^{2}k_{i}-\eta\nabla^{2}\dot{\varphi}(v_{i},v_{i})+\mathcal{B}(v_{i},v_{i})+2(-1)^{i+1}\frac{Q^{2}}{k_{1}-k_{2}} \text{ in } \Sigma-\mathcal{U},$$
where $\mathcal{B}$ is the bilinear form defined in Lemma \ref{laplacianalturangulo} and $$Q^{2}=h_{12,1}^{2}+h_{12,2}^{2}=h_{11,2}^{2}+h_{22,1}^{2}.$$
\end{lemma}
\begin{proof}
We only prove the formula for the first principal curvature $k_{1}$ because the reasoning for $k_{2}$ is the same. 
Fix any point $p\in\Sigma-\mathcal{U}$ and consider a geodesic frame $\{u_{1},u_{2}\}$ of $T_{p}\Sigma$. Then,
\begin{equation}
\label{deltak}
\Delta k_{1}=\sum_{i=1}^{2}\langle\nabla_{u_{i}}\nabla k_{1},u_{i}\rangle=\sum_{i=1}^{2}\langle\nabla_{u_{i}}\nabla\mathcal{S}(v_{1},v_{1}),u_{i}\rangle.
\end{equation}
From  item 1. of Lemma \ref{paso 1}, $\mathcal{S}(\nabla_{u_{i}}v_{1},v_{1})=0$ and we have,
\begin{equation}
\label{nablaA}
\nabla\mathcal{S}(v_{1},v_{1})=\sum_{i=1}^{2}((\nabla_{u_{i}}\mathcal{S})(v_{1},v_{1}))u_{i}.
\end{equation}
By using  \eqref{nablaA} and \eqref{deltak}, we prove that
\begin{equation}
\Delta k_{1}=\sum_{i=1}^{2}\langle\nabla_{u_{i}}(\nabla_{u_{i}}\mathcal{S})(v_{1},v_{1})u_{i} ,u_{i} \rangle=(\Delta\mathcal{S})(v_{1},v_{1})+2\frac{Q^{2}}{k_{1}-k_{2}}
\end{equation}
and the Lemma follows from item 8. of Lemma \ref{laplacianalturangulo}.
\end{proof}
\begin{lemma}
\label{equations}
Let $\Sigma$ be a $[\varphi,\vec{e}_{3}]$-minimal immersion in $\mathbb{R}^{3}_\alpha$ with  $k_1<0$, $H=k_1+k_2<0$. If for any positive smooth function $\psi:\Sigma \rightarrow ]0,+\infty[$ take  the operator $${\cal J}^\psi:= \Delta^{\varphi + 2 \log \psi},$$ then on $\Sigma \backslash {\cal U}$ we have 
\begin{align}
{\cal J}^{\eta}\left(\frac{k_{2}}{\eta}\right)&=-\dddot{\varphi}\langle\nabla\mu,v_{2}\rangle^{2}+\ddot{\varphi}\left(\frac{k_{2}}{\eta}\right)(1+2\langle\nabla\mu,v_{2}\rangle^{2})\label{paso3} 
  -\frac{2}{\eta}\frac{Q^{2}}{k_{1}-k_{2}}. \\
{\cal J}^{-k_1}\left(\frac{\eta}{k_{1}}\right)&=\dddot{\varphi}\langle\nabla\mu,v_{1}\rangle^{2}\left(\frac{\eta}{k_{1}}\right)^{2}-\ddot{\varphi}\left(\frac{\eta}{k_{1}}\right)(1+2\langle\nabla\mu,v_{1}\rangle^{2}) \label{paso4}\\
 &-2\left(\frac{\eta}{k_{1}} \right)\frac{Q^{2}}{k_{1}(k_{1}-k_{2})} .\nonumber
\end{align}
In particular, if $\dddot{\varphi}\leq 0$ on $]\alpha,+\infty[$ and  $\varphi$ satisfies \eqref{c1}, then
\begin{equation}
\label{pstrongmax}
{\cal J}^{{\eta}}\left(\frac{k_{2}}{\eta}\right)\geq 0 \text{ on } \{p\in\Sigma: k_{2}(p)>0\}.
\end{equation}
\end{lemma}
\begin{proof}
It is not difficult to see that
\begin{equation}
\label{regder}
{\cal J}^\eta\left(\frac{k_{2}}{\eta}\right)=\frac{\eta\Delta^{\varphi}k_{2}-k_{2}\Delta^{\varphi}\eta}{\eta^{2}} \ \ \text{and} \ \ {\cal J}^{-k_{1}}\left(\frac{\eta}{k_{1}}\right)=\frac{k_{1}\Delta^{\varphi}\eta-\eta\Delta^{\varphi}k_{1}}{k_{1}^{2}}.
\end{equation}
Moreover, from Lemma \ref{laplacianalturangulo} and Lemma \ref{paso 2}, we get that
\begin{align}
\label{equ1}
\eta\Delta^{\varphi}k_{i}&=-\vert\mathcal{S}\vert^{2}k_{i}\eta-\eta^{2}(\dddot{\varphi}\langle\nabla\mu,v_{i}\rangle^{2}-\ddot{\varphi}\eta k_{i})+2\ddot{\varphi}k_{i}\langle\nabla\mu,v_{i}\rangle^{2}\\
&-2(-1)^{i+1}\eta\frac{Q^{2}}{k_{1}-k_{2}},\nonumber\\
\label{equ2}k_{2}\Delta^{\varphi}\eta&=-\ddot{\varphi}\eta k_{2}\vert\nabla\mu\vert^{2}-\vert\mathcal{S}\vert^{2}k_{2}\eta, \\
\label{equ3}
k_{1}\Delta^{\varphi}\eta&=-\ddot{\varphi}\eta k_{1}\vert\nabla\mu\vert^{2}-\vert\mathcal{S}\vert^{2}k_{1}\eta,
\end{align}
and  we may conclude from  \eqref{regder}, \eqref{equ1}, \eqref{equ2} and \eqref{equ3} by  a straightforward computation.
\end{proof}
\begin{lemma}
\label{oydrift}
Let $\Sigma$  be a properly embedded $[\varphi,\vec{e}_{3}]$-minimal surface without boundary in $\R^3_\alpha$ with  $\varphi:\mathbb{R}\rightarrow\mathbb{R}$ satisying \eqref{c1} and \eqref{cc3}. Then, $ \Sigma$ is complete and the generalized  Omori-Yau maximum principle can be applied to $\Delta^\varphi$.
\end{lemma}
\begin{proof}
Consider the function $\gamma:\Sigma\longrightarrow \R$ given by $\gamma(p) = 2 \log \vert p\vert$, then as $\Sigma$ is properly embedded and $\varphi$ satisfies \eqref{c1} and \eqref{cc3},  we have 
\begin{align}
&\gamma(p) \rightarrow +\infty \quad \text{ as $p\rightarrow \infty$} \label{comp1}\\
&\vert \nabla \gamma(p)\vert = 2\frac{\vert p^\top\vert}{\vert p\vert^2} \leq 2,\quad \vert p\vert\gg 0\label{comp2}\\
& \Delta^\varphi \gamma(p) = -4\frac{\vert p^\top\vert^2}{\vert p\vert^4} + \frac{2 \mu(p) \dot{\varphi}(\mu(p))+4}{\vert p\vert^2} \leq 2 A + 1, \quad \vert p\vert\gg 0.\label{comp3}
\end{align}
and from Theorem \ref{oy} we can apply the generalizad Omori-Yau maximum principle to $\Delta^\varphi$.

By taking $\gamma$ along any  divergent geodesic, it is clear from \eqref{comp1} and \eqref{comp2} that any properly embedded surface in $\R^3$ is complete.\end{proof}
\subsection*{{\sc Proof of Theorem B}}
 
Let $\Sigma$  be a properly embedded $[\varphi,\vec{e}_{3}]$-minimal surface in $\R^3_\alpha$ with  $\varphi:\mathbb{R}\rightarrow\mathbb{R}$ satisying \eqref{c1}, \eqref{cc3} and $\dddot{\varphi}\leq 0$ on $]\alpha,+\infty[$. Then from Theorem \ref{tanprinH}, we can assume that $\eta>0$ everywhere. Take $k_1<0$, $k_1\leq k_2$, $H=k_1+k_2<0$.
\setcounter{theorem}{0}

\numberwithin{theorem}{section}

We only need to prove that $\Sigma$ is convex provided $ \Lambda K$ is bounded from below since the converse is trivial. For proving that,  we will argue by contradiction and suppose there exists a point $p_0\in\Sigma$ such that $K(p_0)<0$. Then,
\begin{equation} 
0 < \vartheta := \sup_\Sigma\frac{k_2}{\eta}=\sup_{\Omega^+}\frac{k_2}{\eta},\label{sup}
\end{equation}
where $\Omega^+ = \{p\in \Sigma \ \vert \  k_2(p)>0\}$.
\begin{claim}\label{cl1} The supremum $\vartheta$ is not attained.
\end{claim}
\begin{proof}
Suppose it is attained at a point $p$, then  from \eqref{pstrongmax} and 
the strong maximum principle, see \cite[Theorem 3.5]{GT}, $\frac{k_2}{\eta}$ is constant on $\Sigma$ and $Q\equiv 0$.  Thus, from Lemma \ref{paso 1}, $\{v_1,v_2\}$ is parallel and  then $k_1 k_2 \equiv 0$, getting a contradiction with  \eqref{sup}.
\end{proof}
\begin{claim}\label{cl0}
If $\{p_n\} \subset\ \Omega^{+}$ is a sequence of points such that $\frac{k_2}{\eta}(p_{n})\rightarrow \vartheta$, then after passing to a subsequence, $\mu(p_n)\rightarrow +\infty$ and
\begin{enumerate}
\item if $\Lambda=0$ and $\frac{\eta}{k_1}(p_n) \rightarrow 0$, then $\eta(p_n) \rightarrow 0$.
\item if $\Lambda \neq 0$, then $\eta(p_n) \rightarrow 0$ and $\frac{\eta}{k_1}(p_n) \rightarrow 0$.
\end{enumerate} 
\end{claim}
\begin{proof}
From \eqref{cc3}, the function $2\ddot{\varphi}-\dot{\varphi}^{2}$ is upper bounded on $]\alpha,+\infty[$ and  we can apply the Theorem \ref{boundnessarea}, getting that the sequence $\Sigma_{n}=\Sigma - p_n$ has area uniformly bounded on compact subsets of $\R^3_\alpha$.

\

Each $\Sigma_{n}$ is a $[\varphi_{n},\vec{e}_{3}]$-minimal surface with 
\begin{equation}\varphi_{n}(u)=\varphi(u+\mu(p_{n}))- \varphi(\mu(p_n)),\quad  \text{ for each $n\in\mathbb{N}$}.\label{varphinn}\end{equation}
If $\sup_n\{\mu(p_{n})\}<+\infty$ then, by taking an accumulation point $\mu_\infty$ of $\{\mu(p_{n})\}$ and applying the compactness Theorem \ref{CompacWhite}, we get that, after passing to a subsequence, $\mu(p_{n})\rightarrow\mu_{\infty}\in\mathbb{R}$,  $\Sigma_{n}$ converges  to a properly embedded $[\varphi_{\infty},\vec{e}_{3}]$-minimal surface $\Sigma_{\infty}$ with $H\leq 0$ where  $\varphi_{\infty}(u):=\varphi(u+\mu_{\infty})- \varphi(\mu_\infty)$.  From Theorem A,  the length of second fundamental form of $\Sigma_n$  is bounded, therefore the convergence must be  smooth at the origin and so the function $\frac{k_2}{\eta}$ reachs its supremum at the origin. This is a contradiction with Claim \ref{cl1}.

If $\Lambda=0$ and $\frac{\eta}{k_1}(p_n) \rightarrow 0$, we consider $\Sigma'_n = \lambda_n(\Sigma - p_n)$ where $\lambda_n= -\frac{k_1}{\eta}(p_n)$ then, from \eqref{c1} and \eqref{cc3} and  after passing to a subsequence, we get 
that 
$$\eta(p_n)\rightarrow \eta_\infty, \quad  \frac{\dot{\varphi}(\mu(p_n))}{\lambda_n} = 1+\frac{k_2}{k_1}(p_n)\rightarrow 0.$$
Applying Theorem \ref{blowup},  $\Sigma'_n $ converge smoothly to a plane $\Sigma_{\infty}$, with principal curvatures at the origin given by
$$
k_1 = - \eta_{\infty} \,\textnormal{ and }\, k_2 =  \eta_{\infty},
$$
which implies that $\eta_{\infty} = 0$.

If $\Lambda \neq 0$, since $\frac{k_1}{\eta}+\frac{k_2}{\eta} = -\dot{\varphi}$, we have that $\frac{\eta}{k_1}(p_n) \rightarrow 0$. Let us suppose by contradiction that $\eta(p_n) \rightarrow \eta_{\infty} \neq 0$. Then $k_1(p_n) \rightarrow -\infty$ and $k_2(p_n) \rightarrow \vartheta$, getting to a contradiction with the hypothesis that $\Lambda K$ is bounded from below.
\end{proof}
We will distinguish the case that $\dot{\varphi}$ is bounded ($\Lambda=0$) from the unbounded case ($\Lambda\neq 0$):
\subsection*{{\sc $\bullet$  Case $\Lambda=0$.}}
In this case, from \eqref{cc3} we have  that  on $]\alpha,+\infty[$,
\begin{equation} \label{cotaB}
0<\dot{\varphi}<\sup_{]\alpha,+\infty[} \dot{\varphi}= \beta.
\end{equation} 

\begin{claim}\label{cl2}
The case $\vartheta=+\infty$ is not possible.\end{claim}
\begin{proof}
Assume  there exists a sequence of points $\{p_{n}\}$ such that $\frac{k_{2}}{\eta}(p_{n})\rightarrow +\infty$.
Using that 
\begin{equation}
\label{relacion1}
\left(\frac{k_{1}}{k_{2}} \right)+1=-\dot{\varphi}\left(\frac{\eta}{k_{2}} \right),\quad \frac{k_1+k_2}{\eta}=-\dot{\varphi},
\end{equation}
we get   $(k_{1}/ k_{2})(p_{n})\rightarrow-1$  and  $(\eta/k_{1})(p_{n})\rightarrow 0$. In particular, 
\begin{equation}
\label{supremo}
\tau = \sup_{\Sigma}\frac{\eta}{k_{1}}=0.
\end{equation}
and  $\tau$  is not attained at a interior point.  Now we may apply the generalized Omori-Yau maximum principle for $\Delta^\varphi$ and conclude that  there exists a sequence of points, $\{q_n\}\subset \Sigma$,  $\vert q_{n}\vert \rightarrow +\infty$,  such that
\begin{equation}
\label{OmoriYau}
\frac{\eta}{k_{1}}(q_{n})\rightarrow 0 \ \ , \ \ \vert\nabla\left(\frac{\eta}{k_{1}}\right)\vert(q_{n})\rightarrow 0 \ \ \text{and} \ \ \Delta^\varphi\left(\frac{\eta}{k_{1}} \right)(q_{n})\leq 0.
\end{equation}
Consequently, it follows from    \eqref{relacion1}, and \eqref{OmoriYau}  that $(k_{2}/\eta)(q_{n})\rightarrow +\infty$ and so, for $n$ large enough $\{q_{n}\}\subset\Omega^+$. In particular  there exists $n_0\in\mathbb{N}$ such that \eqref{paso3} and \eqref{paso4} hold for $n \geq n_0$. For the rest of the proof of Theorem B, any statement that some quantity tends to a limit refers only to the quantity at the corresponding points.

\

Now, from Claim \ref{cl0}, after passing to subsequence, $\mu \rightarrow +\infty$, $\eta\rightarrow 0$ and  $\frac{k_2}{k_1}=-\frac{\dot{\varphi} \, \eta}{k_1} - 1 \rightarrow -1$.
 Thus, form   Lemma \ref{laplacianalturangulo}, we have
\begin{equation}
\label{lim}
\left\vert \frac{\nabla \eta}{k_1} \right\vert^2 = \langle \nabla \mu ,v_1\rangle^2 + \left(\frac{k_2}{k_1}\right)^2 \langle \nabla \mu ,v_2\rangle^2 \longrightarrow 1.
\end{equation}
and  by \eqref{OmoriYau} and \eqref{lim},
\begin{equation}
\label{Campo}
\frac{\nabla \eta}{k_1} \rightarrow \mathcal{X}, \qquad \frac{\eta}{k_1}\frac{ \nabla k_1}{k_1}\rightarrow \mathcal{X}, \quad \mathcal{X}\neq 0
\end{equation}
Since 
\begin{equation}\frac{\eta}{k_1}\frac{\nabla H}{k_1} = \frac{\eta}{k_1}\frac{\nabla k_1}{k_1}+\frac{\eta}{k_1}\frac{\nabla k_2}{k_1} = -\frac{\eta^{2}}{k_1}\frac{\nabla \dot{\varphi}}{k_1} -\frac{\eta}{k_1}\frac{\dot{\varphi}\nabla \eta}{k_1},\label{campo2}\end{equation}
it follows from  Lemma \ref{laplacianalturangulo},  \eqref{Campo} and \eqref{campo2} that 
$$\frac{\eta}{k_1}\frac{h_{11,2}}{k_1} \rightarrow <\mathcal{X},v_2>, \qquad \frac{\eta}{k_1}\frac{h_{22,1}}{k_1} \rightarrow -<\mathcal{X},v_1>. $$
In particular,
\begin{equation}
\frac{\eta^2}{k_1^4}Q^2 \rightarrow \vert \mathcal{X} \vert =1\label{qq}
\end{equation}

Multiplying by $(\eta/k_{1})$ in \eqref{paso4}, we obtain
\begin{align}
\label{eq5}
&\left(\frac{\eta}{k_{1}}\right)\Delta^\varphi \left(\frac{\eta}{k_{1}}\right)+2\left(\frac{\eta}{k_{1}}\right)\langle\nabla\left(\frac{\eta}{k_{1}}\right),\frac{\nabla k_{1}}{k_{1}}\rangle
= 
\dddot{\varphi}\langle\nabla\mu,v_{1}\rangle^{2}\left(\frac{\eta}{k_{1}}\right)^{3}\\
&-\ddot{\varphi}\left(\frac{\eta}{k_{1}}\right)^{2}(1+2\langle\nabla\mu,v_{1}\rangle^{2})-2k_{1}\left(\frac{\eta^{2}}{k_{1}^{4}} \right)\frac{Q^{2}}{k_{1}-k_{2}}.\nonumber
\end{align}
Using that $k_2/k_1 \rightarrow -1$,  \eqref{c1}, \eqref{cc3}, \eqref{OmoriYau} and \eqref{qq}, we can take  limit when $n\rightarrow +\infty$ in the above equality to get:
$$ 0\leq -1,$$
a contradiction.
\end{proof}
\begin{claim}\label{cl3}
If  $\{p_n\} \subset \Sigma$ is a sequence of points such that $\frac{k_2}{\eta}\rightarrow \vartheta < +\infty$, then after passing to a subsequence,
$$
\mu\rightarrow +\infty,\quad \eta \rightarrow  0, \quad \frac{k_1}{k_2}\rightarrow -\frac{\beta}{\vartheta} - 1.$$\end{claim}
\begin{proof} 
By taking  $\Sigma_{n}=\Sigma - p_n$, we can argue as in the first part of Claim \ref{cl0} to prove that after passing to a subsequence, $\mu\rightarrow +\infty.$
Then,  from \eqref{varphinn}, 
$$\varphi_{n}\rightarrow\varphi_\infty,  \quad \text{ with \ \ $\varphi_\infty(u) = \beta \,u$},$$
and using again  the compactness Theorem \ref{CompacWhite}, after passing to a subsequence, we have that $\Sigma_{n}$ converges  to a properly embedded translating soliton $\Sigma_{\infty}$ containing the origin with $H\leq 0$. But from Theorem A,   the length of the second fundamental form of $\Sigma_n$ is bounded, so the convergence is smooth and we conclude that if   $\Sigma_{\infty}$   is not a vertical plane, $k_2/\eta$ attains it supremum value at the origin of $\Sigma_\infty$ which contradicts Claim \ref{cl1}.

Thus, $\eta\rightarrow 0$ and 
$$ \frac{k_1}{k_2}= \frac{H}{k_2}-1= -\dot{\varphi} \frac{\eta}{k_2} -1 \rightarrow \frac{\beta}{\vartheta} -1.$$

\end{proof}
\begin{claim} \label{cl4}The case $0<\vartheta<+\infty$ is not possible.
\end{claim}
\begin{proof}
If
$$0< \vartheta = \sup_\Sigma \frac{k_2}{\eta} = \sup_{\Omega^+} \frac{k_2}{\eta} < \infty, $$ 
them  from Lemma \ref{oydrift}, Theorem \ref{oy} and   Claim \ref{cl1},  there exists a sequence of points $\{p_{n}\}\subset\Omega^+$ , $\vert p_n\vert \rightarrow+\infty$  such that
\begin{equation}
\label{OmoriYau2}
\left(\frac{k_{2}}{\eta}\right)\rightarrow\vartheta, \qquad \left\vert\nabla\left(\frac{k_{2}}{\eta}\right)\right\vert\rightarrow 0, \qquad \Delta^\varphi\left(\frac{k_{2}}{\eta} \right)(p_{n})\leq 0.
\end{equation}
From  Claim \ref{cl3}  we get 
$$ \nabla \mu \rightarrow \vec{e}_3, \quad \mu \rightarrow +\infty, \quad \frac{k_1}{k_2}\rightarrow -\frac{\beta}{\vartheta} -1,$$
and from Lemma \ref{laplacianalturangulo}, 
$$ \left\vert \frac{k_2}{\eta}\frac{\nabla \eta}{\eta} \right\vert^2 = \frac{k_2^4}{\eta^4}\left(\frac{k_1^2}{k_2^2}\langle \nabla \mu ,v_1\rangle^2 +  \langle \nabla \mu ,v_2\rangle^2 \right)\longrightarrow C\neq 0,  $$
where $C$ is a constant such that $C\in [\vartheta^4,2 \vartheta^4+ \vartheta^2(\beta^2 + 2 \beta\vartheta)[$. Then, by \eqref{OmoriYau2},
\begin{equation}
\frac{\nabla k_2}{\eta} \rightarrow \mathcal{Y}, \qquad \frac{k_2}{\eta}\frac{ \nabla \eta}{\eta}\rightarrow \mathcal{Y}, \quad \mathcal{Y}\neq 0\label{campo}
\end{equation}
Arguing as in Claim \ref{cl3} we can prove that
\begin{equation}
\label{calculoQ22}
\frac{h_{11,2}}{\eta}\rightarrow -\vartheta^2\left(\frac{\beta}{\vartheta}+ 1\right) \langle \vec{e}_3,v_2\rangle, \qquad \frac{h_{22,1}}{\eta}\rightarrow -\vartheta^2\left(\frac{\beta}{\vartheta}+ 1\right)  \langle \vec{e}_3,v_1\rangle,
\end{equation}
and then, 
\begin{equation}
\frac{Q^2}{\eta^2} =\frac{h_{11,2}^2 + h_{22,1}^2}{\eta^2}\rightarrow  \vartheta^4\left(\frac{\beta}{\vartheta}+ 1\right)^2\label{qq2}
\end{equation}

Multiplying by $(k_2/\eta)$ in \eqref{paso3}, we obtain
\begin{align}
\label{expression2}
&\frac{k_{2}}{\eta}\Delta^\varphi\left(\frac{k_{2}}{\eta}\right)+2\, \frac{k_{2}}{\eta}\langle\nabla\left(\frac{k_{2}}{\eta}\right),\frac{\nabla\eta}{\eta}\rangle
=-\dddot{\varphi}\frac{k_{2}}{\eta}\langle\nabla\mu,v_{2}\rangle^{2}\\
&+\ddot{\varphi}\left(\frac{k_{2}}{\eta}\right)^{2}(1+2\langle\nabla\mu,v_{2}\rangle^{2})-2\left(\frac{Q^{2}}{\eta^{2}}\right)\frac{k_{2}}{k_{1}-k_{2}}.\nonumber
\end{align}

Using that $\frac{k_1}{k_2}\rightarrow -\frac{\beta}{\vartheta}-1$,  \eqref{c1}, \eqref{cc3}, \eqref{OmoriYau2} and \eqref{qq2}, we can take  limit when $n\rightarrow +\infty$ in the above equality to get
$$ 0\geq 2\frac{ \vartheta^4\left(\frac{\beta}{\vartheta}+ 1\right)^2}{\frac{\beta}{\vartheta} + 2} > 0$$
which is  contradiction.
\end{proof}

\subsection*{{\sc $\bullet$  Case $\Lambda\neq 0$.}}
As the supremum of $k_2/\eta$ is not attained on $\Sigma$, we can take any divergent sequence of points  $\{p_n\}\subset \Omega^+$  such that $k_2/\eta \rightarrow \vartheta$.
\begin{claim}\label{cl5}
If $\Lambda\neq 0$ and  $\{p_n\} \subset \Sigma$ is a sequence of points such that $\frac{k_2}{\eta}\rightarrow \vartheta < +\infty$, then after passing to a subsequence 
$$
\frac{\eta}{k_1}\rightarrow 0, \quad \mu\rightarrow +\infty,\quad \eta \rightarrow  0, \quad \frac{k_2}{k_1}\rightarrow 0.$$\end{claim}
\begin{proof}
By taking  $\Sigma_{n}=\Sigma - p_n$, we can argue as in the first part of Claim \ref{cl0} to prove that after passing to a subsequence, $\mu\rightarrow +\infty.$   Since  $\frac{k_1+k_2}{\eta} = -\dot{\varphi}$, we have that $\frac{\eta}{k_1}\rightarrow 0$ and Claim \ref{cl0} gives  that, after passing to subsequence,  $\eta\rightarrow 0$. 
Finally, 
$$\frac{k_1}{k_2} = \frac{H}{k_2} - 1 = -\dot{\varphi} \frac{\eta}{k_2}\rightarrow -\infty.$$
\end{proof}

\begin{claim}\label{cl6}
The case $0<\vartheta <+\infty$ is not possible.
\end{claim}
\begin{proof}
From Theorem \ref{oy}  and Claim \ref{cl1} there exists a sequence of points $\{q_{n}\}\subset\Omega^+$ , $\vert q_n\vert \rightarrow+\infty$  such that
\begin{equation}
\label{OmoriYau3}
\frac{k_{2}}{\eta}(q_{n})\rightarrow\vartheta, \qquad \left\vert\nabla\left(\frac{k_{2}}{\eta}\right)\right\vert(q_{n})\rightarrow 0, \qquad \Delta^\varphi\left(\frac{k_{2}}{\eta} \right)(q_{n})\leq 0.
\end{equation}
By an straightforward computation we obtain
\begin{align*}
\frac{\eta^2}{k_1k_2^2}\nabla k_2 & = \frac{\eta^3}{k_1 k_2^2}\nabla \left(\frac{k_2}{\eta}\right) + \frac{\eta}{k_1k_2} \nabla \eta, \\
\frac{\eta^2}{k_1k_2^2}\nabla k_1 & = -\frac{\eta^3}{k_1 k_2^2}\nabla \left(\frac{k_2}{\eta}\right)+ \frac{\eta}{k_2^2} \nabla \eta - \frac{\eta^3 \ddot{\varphi}}{k_1 k_2}\nabla \mu
\end{align*}
and using  Claim \ref{cl5} and \eqref{OmoriYau3},
\begin{align*}
\frac{\eta^2}{k_1k_2^2}h_{22,1}\rightarrow \frac{1}{\vartheta}  \ \langle \vec{e}_3,v_1\rangle, \qquad \frac{\eta^2}{k_1k_2^2}h_{11,2}\rightarrow \frac{1}{\vartheta} \ \langle \vec{e}_3,v_2\rangle.
\end{align*}
which gives
\begin{equation}\label{qqq3}
\frac{\eta^4}{k_1^2 k_2^4}Q^2 =\frac{\eta^4}{k_1^2 k_2^4}(h_{11,2}^2 + h_{22,1}^2)\rightarrow \frac{1}{\vartheta^{2}}  > 0.
\end{equation}
As the equation \eqref{paso3} holds on $\Omega^+$, multiplying by $\frac{\eta^{3}}{k_{1}k_{2}^{2}}$  we get that
\begin{align}
\label{contr}
&\frac{\eta^{3}}{k_{1}k_{2}^{2}}\Delta^{\varphi}\left(\frac{k_{2}}{\eta}\right)+2\, \frac{\eta^{3}}{k_{1}k_{2}^{2}}\langle\nabla\left(\frac{k_{2}}{\eta}\right),\frac{\nabla\eta}{\eta}\rangle=-\dddot{\varphi}\frac{\eta^{3}}{k_{1}k_{2}^{2}}\langle\nabla\mu,v_{2}\rangle^{2}\\
&+\ddot{\varphi}\frac{\eta^{3}}{k_{1}k_{2}^{2}}\left(\frac{k_{2}}{\eta}\right)(1+2\langle\nabla\mu,v_{2}\rangle^{2})-2\frac{\eta^{3}}{k_{1}k_{2}^{2}}\frac{1}{\eta}\frac{Q^{2}}{k_{1}-k_{2}}.\nonumber
\end{align}

\end{proof}
\begin{claim}\label{cl7}
The case $\vartheta =+\infty$ is not possible.
\end{claim}
\begin{proof}
Assume by contradiction that $\vartheta=+\infty$. Let   $g:\R\longrightarrow ]-1,1[$ be a bounded smooth function satisfying:
\begin{align}\label{gg}
& \dot{g} \geq 0, \quad \text{ on $\R$},\\
& g(x) = 1 - \frac{1}{x},  \quad \text{ on $[1,+\infty[.$}
\end{align}
Let  $h:\Sigma\longrightarrow \R$ be  the function $h(p) = g\left(\frac{k_2}{\eta}(p)\right)$.  Using  \eqref{paso3}, a straightforward computation provides,
\begin{align}
\label{drifh}
\Delta^{\varphi}h&+2\langle\nabla h,\frac{\nabla\eta}{\eta}\rangle
=\ddot{g}\,\vert\nabla\left(\frac{k_{2}}{\eta}\right)\vert^{2}
-\dot{g}\,\dddot{\varphi}\langle\nabla\mu,v_{2}\rangle^{2}\\&+\dot{g}\,\ddot{\varphi}\left(\frac{k_{2}}{\eta} \right)(1+2\langle\nabla\mu,v_{2}\rangle^{2})-2\,\frac{\dot{g}}{\eta}\frac{Q^{2}}{k_{1}-k_{2}}.\nonumber
\end{align}
Since $\vartheta=+\infty$, it is clear that 
\begin{equation}
\label{suph}
\sup_{\Sigma}\{ h\}=1,
\end{equation}
and it is not attained on $\Sigma$. Now, from Lemma \ref{oydrift}  we can apply the Theorem \ref{oy} and  there exists a divergent sequence $\{q_{n}\}$ such that
\begin{equation}
\label{OmoriYauh}
h \rightarrow 1 , \quad \vert\nabla h\vert\rightarrow 0,  \ \ \text{and} \ \ \Delta^{\varphi}h(q_{n})\leq 0.
\end{equation}
Thus, $\frac{k_2}{\eta}\rightarrow +\infty$, $\frac{\eta}{k_1}\rightarrow 0$ and, from Claim \ref{cl0}, after passing to a subsequence we  have also  that $\mu\rightarrow +\infty$ and $\eta\rightarrow 0$. Now, we can argue as in Claim \ref{cl6} to get that 
\begin{align*}
\frac{\eta}{k_1k_2}h_{22,1}\rightarrow  \ \langle \vec{e}_3,v_1\rangle, \qquad \frac{\eta}{k_1k_2}h_{11,2}\rightarrow  \ \langle \vec{e}_3,v_2\rangle.
\end{align*}
which gives
\begin{equation}\label{qqq4}
\frac{\eta^2}{k_1^2 k_2^2}Q^2 =\frac{\eta^2}{k_1^2 k_2^2}(h_{11,2}^2 + h_{22,1}^2)\rightarrow 1>0.
\end{equation}
Using that    $\frac{k_2}{k_1} \in ]-1,0[$ in $\Omega^+$ and that for  $n$ large enough, 
$$\dot{g}\left(\frac{k_2}{\eta}(p_n)\right)=\frac{\eta^2}{k_2^2}(p_n), \qquad \ddot{g}\left(\frac{k_2}{\eta}(p_n)\right)=-2\frac{\eta^3}{k_2^3}(p_n).$$
If we multiply by $\frac{\eta}{k_{1}}$ in the expression \eqref{drifh}  take limit when $n\rightarrow +\infty $, we get 
$$ 0 \leq -\frac{2}{1-C}< 0,$$
where $\frac{k_2}{k_1}\rightarrow C \in[-1,0]$, which is a contradiction.
\end{proof}
From the above Claims, the only possibility is that $\vartheta \leq 0$,which  concludes the proof.\hfill $\square$

\

From Theorem \ref{blowup}, Theorem B and arguing as in \cite[Corollary 2.3]{HIMW2}, we may obtain
\begin{corollary}
Let    $\Sigma$ be  as in Theorem B with $\Lambda K$ bounded from below. If $\{p_n\}$ is any divergent sequence in $\Sigma$ and    $\{\lambda_{n}\}\rightarrow+\infty$ any sequence   such that  $\{\dot{\varphi}(\mu(p_{n}))/\lambda_{n}\}\rightarrow C$ for some constant $C> 0$.  Then,   $\Sigma_n=\lambda_n ( \Sigma-p_n)$  converge smoothly (after passing to a subsequence) to  a vertical plane, a grim reaper surface, or 
a  titled grim reaper surface.
\end{corollary}
Moreover, from Theorem \ref{tanprinK} and Theorem B, we have
\begin{corollary}
Let   $\Sigma$ be  as in Theorem B with $\Lambda K$ bounded from below. If $K$ vanishes anywhere, then $\Sigma$ has vanishing curvature.\end{corollary}
\subsubsection*{Some interesting questions} 
We conclude this paper with two  questions related to our Theorem B,  the first is whether an entire $[\varphi,\vec{e}_3]$-minimal  vertical graph in $\R^3$ with $\varphi$ satisfying \eqref{c1} and \eqref{cc3} is convex.  The second is whether an entire  $[\varphi,\vec{e}_3]$-minimal  vertical graph in $\R^3$ with  $H(p)\rightarrow 0$  as $\vert p\vert\rightarrow \infty$ and  $\varphi$ satisfying \eqref{c1} and \eqref{cc3} is rotationally symmetric.  We expect affirmative answers to both questions.

\end{document}